\documentclass[11pt]{article}
\usepackage{amsmath,amssymb,amsthm}
\usepackage{graphicx,subfigure,float,url}
\usepackage{pdfsync}
\usepackage[english]{babel}
\usepackage[utf8]{inputenc}
\usepackage{tipa}
\usepackage{tikz}

\newlength\imagewidth
\newlength\imagescale

\usepackage{bbold}
\usepackage{bbm}

\usepackage{fancybox}
\usepackage{listings}
\usepackage{mathtools}
\usepackage{amsfonts}
\usepackage[bottom]{footmisc}
\usepackage{verbatim}

\usepackage{colortbl}

\usepackage{float}
\usepackage{makeidx}

\normalbaselines

\newtheorem{Theorem}{Theorem}[section]

\newtheorem{Proposition}[Theorem]{Proposition}

\newtheorem{Lemma}[Theorem]{Lemma}

\newtheorem{Corollary}[Theorem]{Corollary}

\newtheorem{Remark}[Theorem]{Remark}

\newcommand{\RR}{{{\rm I} \kern -.15em {\rm R} }}

\newcommand{\C}{{{\rm l} \kern -.50em {\rm C} }}

\newcommand{\nat}{{{\rm I} \kern -.15em {\rm N} }}

\newcommand{\be}{\begin{equation}}
\newcommand{\ee}{\end{equation}}
\newcommand{\beq}{\begin{eqnarray}}
\newcommand{\eeq}{\end{eqnarray}}
\newcommand{\beqs}{\begin{eqnarray*}}
\newcommand{\eeqs}{\end{eqnarray*}}
\newcommand{\bt}{\begin{Theorem}}
\newcommand{\et}{\end{Theorem}}
\newcommand{\br}{\begin{Remark}}
\newcommand{\er}{\end{Remark}}
\newcommand{\bc}{\begin{Corollary}}
\newcommand{\ec}{\end{Corollary}}
\newcommand{\bl}{\begin{Lemma}}
\newcommand{\el}{\end{Lemma}}
\newcommand{\bd}{\begin{definition}}
\newcommand{\ed}{\end{definition}}

\renewcommand{\geq}{\geqslant}
\renewcommand{\leq}{\leqslant}

\newcommand{\W}{\mathcal{H}}

\graphicspath{{figures/}}

\title{Well-posedness and stability for semilinear wave-type equations with time delay}
\author{
\and
Alessandro Paolucci\footnote{Dipartimento di Ingegneria e Scienze dell'Informazione e Matematica, Universit\`{a} di L'Aquila, Via Vetoio, Loc. Coppito, 67010 L'Aquila Italy (\texttt{alessandro.paolucci2@graduate.univaq.it}).}
\and
Cristina Pignotti\footnote{Dipartimento di Ingegneria e Scienze dell'Informazione e Matematica, Universit\`{a} di L'Aquila, Via Vetoio, Loc. Coppito, 67010 L'Aquila Italy (\texttt{cristina.pignotti@univaq.it}).}
}

\date{}

\begin{document}

\textwidth=160 mm

\textheight=225mm

\parindent=8mm

\frenchspacing

\maketitle

\begin{abstract}
In this paper we analyze a semilinear abstract damped wave-type equation with time delay. We assume that the delay feedback coefficient is variable in time and belonging to $L^1_{loc}([0, +\infty)).$ Under suitable assumptions, we show well-posedness and exponential stability for \emph{small} initial data. Our strategy combines careful energy estimates and continuity arguments.
Some examples illustrate the abstract results.
\end{abstract}

\vspace{5 mm}

\def\qed{\hbox{\hskip 6pt\vrule width6pt
height7pt
depth1pt  \hskip1pt}\bigskip}

\section{Introduction}
Let $H$ be an Hilbert space and let $A:{\mathcal D}(A)\subset H \to H$ be a positive self-adjoint operator with dense domain and compact inverse in $H.$ Let us consider the following wave-type equation:
\begin{equation}\label{modello}
\begin{array}{l}
\displaystyle{u_{tt}(t)+A u(t)+CC^*u_t(t)+k(t)BB^* u_t(t-\tau)=\nabla \psi(u(t)),\quad t\geq 0,}\\
\displaystyle{u(0)=u_0, \quad u_t(0)=u_1,}\\ 
\displaystyle{B^*u_t(s)=g(s), \quad s\in [-\tau,0], }
\end{array}
\end{equation}
where $\tau> 0$ is the time delay, the damping coefficients $k(\cdot)$ is a function in $L^1_{loc}([0,+\infty))$ and $(u_0, u_1, g)$ are the initial data in suitable spaces. Moreover, for given real Hilbert spaces $W_1$ and $W_2$ that  will be identified with their dual spaces,  $C:W_1\to H$ and $B:W_2\to H$ are bounded linear operators  with adjoint   $C^*$ and  $B^*$ respectively.  
We assume that the damping operator $CC^*$ satisfies a control geometric property (see e.g. \cite{Bardos} or \cite[Chapter 5]{K}). Moreover, on the delay feedback coefficient, we assume 
\begin{equation}\label{damp_coeff}
\int_{t-\tau}^t |k(s)|ds\leq K, \quad \forall \ t\ge 0,
\end{equation}
for some $K>0.$

Furthermore, $\psi : {\mathcal D}(A^{\frac 1 2})\rightarrow \RR$ is a functional having G\^{a}teaux derivative $D\psi(u)$ at every $u\in {\mathcal D}(A^{\frac 12}).$ In the spirit of \cite{ACS}, we assume the following hypotheses:
\begin{itemize}
\item[{(H1)}] For every $u\in {\mathcal D}(A^{\frac 1 2})$, there exists a constant $c(u)>0$ such that
$$
|D\psi(u)(v)|\leq c(u) ||v||_{{H}} \qquad \forall v\in {\mathcal D}(A^{\frac 1 2}).
$$
Then, $\psi$ can be extended to the whole  $H$ and
 we denote by $\nabla \psi(u)$ the unique vector representing $D\psi(u)$ in the Riesz isomorphism, i.e.
$$
\langle \nabla \psi(u), v \rangle_H =D\psi(u) (v), \qquad \forall v\in H;
$$
\item[ (H2)] for all $r>0$ there exists a constant $L(r)>0$ such that
$$
||\nabla \psi (u)-\nabla \psi (v)||_H \leq L(r) ||A^{\frac 12}(u-v)||_H,
$$
for all $u,v\in {\mathcal D}(A^{\frac 12})$ satisfying $||A^{\frac 12} u||_H\leq r$ and $||A^{\frac 12} v||_H\leq r$.
\item[{ (H3)}] $\psi(0)=0,$  $\nabla \psi(0)=0$ and
there exists a strictly increasing continuous function $h$ such that
\begin{equation}
\label{stima_h}
||\nabla \psi (u)||_H\leq h(||A^{\frac 12} u||_H)||A^{\frac 12}u||_H,
\end{equation}
for all $u\in {\mathcal D}(A^{\frac 12})$.
\end{itemize}
We are interested in studying the well-posedness of system \eqref{modello} and in proving an exponential  stability estimate for solutions corresponding to small initial data, under a suitable assumption involving the model's parameters. A linear version of such a model has been first studied in \cite{SCL12} where a wave equation with frictional damping and delay feedback with constant coefficient has been analyzed, proving an exponential decay estimate under a suitable smallness condition on the delay term coefficient. This result has then  been extended to linear  wave equations with boundary dissipation (see \cite{ANP10}) and with viscoelastic damping (see \cite{AlNP, guesmia}). We quote also  \cite{JEE15, NicaisePignotti18, KP} for related stability results for abstract semilinear evolution equations. However,   in the nonlinear setting, the results previously obtained require that the damping operator $CC^*$ contrasts, in the spirit of \cite{NPSicon06} (cf. also \cite{XYL}), the delay feedback. Indeed, in \cite{JEE15, NicaisePignotti18}, where the delay coefficient is constant, i.e. $k(t)\equiv k,$  in order to have a not increasing energy, it is assumed $\vert k\vert<\frac 1 {\mu}$
and
$$ \Vert B^*u\Vert_{W_2}\le \mu  \Vert C^*u\Vert_{W_1}, \quad \forall u\in {\mathcal D}(A^{\frac 12}). $$
In \cite{KP} the coefficient $k$ is time dependent, as here, but it is assumed
that 
$$
k(t)=k^1(t)+k^2(t),
$$
with  $k^1\in L^1([0,+\infty))$, $k^2\in L^\infty ([0,+\infty))$, and
$||k^2||_\infty$ smaller than a suitable constant depending on the damping operator $CC^*.$
Actually, in \cite{KP}, the model involves a finite number of time delays $\tau_1,\dots, \tau_l.$ Here, for sake of clarity, we consider only a delay feedback. However, our analysis could be easily extended to more than one delay term.

Stability results in presence of delay feedback have also been obtained for specific models, with  $k$ constant, mainly in the linear setting (see e.g. \cite{Ait, AG, AM, Chentouf, Dai, Oquendo, Said}).

The rest of the paper is organized as follows. In section \ref{exp} we rewrite system \eqref{modello} in an abstract form and we give an exponential stability result under an appropriate well-posedness assumption. Then, we show that the well-posedness assumption is satisfied for model \eqref{modello} and so the exponential decay estimate holds for {\em small} initial data. Section \ref{Examples} illustrates some concrete examples for which the abstract theory is applicable.

\section{Exponential stability}
\label{exp}\hspace{5mm}

\setcounter{equation}{0}
Before proving the exponential stability of system \eqref{modello}, we rewrite it in an abstract way. Let us introduce the Hilbert space
$$
\W=\mathcal{D}(A^{\frac 12}) \times H,
$$
endowed with the inner product
$$
\left\langle
\left (
\begin{array}{l}
u\\
v
\end{array}
\right ),
\left (
\begin{array}{l}
\tilde u\\
\tilde v
\end{array}
\right )
\right\rangle_{\W}:= \langle A^{\frac 12}u, A^{\frac 12} \tilde u\rangle_H+\langle v, \tilde v\rangle_H.
$$
If we denote $v(t)=u_t(t)$ and $U(t):=(u(t),v(t))^T$, we can rewrite system \eqref{modello} in the following abstract form
\begin{equation}\label{abstract_form}
\begin{array}{l}
\displaystyle{U'(t)=\mathcal AU(t)- k(t)\mathcal BU(t-\tau)+F(U(t)),}\\
\displaystyle{U(0)=U_0,}\\
\displaystyle{\mathcal BU(t)=f(t), \quad t\in[-\tau,0],}
\end{array}
\end{equation}
where 
$$
\mathcal A=
\begin{pmatrix}
0 & 1 \\ 
-A & -CC^*
\end{pmatrix}, \quad 
\mathcal BU(t)= \left( \begin{array}{l} \hspace{0.7 cm}0 \\ BB^* v(t)\end{array} \right) \quad \text{and} \quad F(U(t))=\left( \begin{array}{l} \hspace{0.65 cm} 0 \\ \nabla \psi(u(t)) \end{array}\right).
$$
We know that, under controllability assumptions on the  damping operator $CC^*$ (see for instance \cite{Bardos, K}), $\mathcal A$ generates an exponentially stable $C_0$-semigroup $\{ S(t)\}_{t\geq 0}$, namely there exist $M,\omega >0$ such that
\begin{equation}\label{decay_semigroup}
||S(t)||_{\mathcal L(\mathcal H)} \leq Me^{-\omega t}, \quad \forall t\geq 0.
\end{equation}
Moreover, the previous hypotheses (H2)-(H3) on $\psi$ imply the following properties on $F$:
\begin{itemize}
\item[(F1)] $F(0)=0$;
\item[(F2)] for any $r>0$ there exists a constant $L(r)>0$ such that
$$
||F(U)-F(V)||_\W\leq L(r) ||U-V||_\W,
$$
whenever $||U||_\W,||V||_\W\leq r$.
\end{itemize}
Let us denote
\begin{equation}\label{normab}
\Vert B\Vert_{{\mathcal L}(\W)}=\Vert B^*\Vert_{{\mathcal L}(\W)}=b.
\end{equation}
Then, by \eqref{normab}, 
\begin{equation}\label{normabcal}
\Vert\mathcal B\Vert_{{\mathcal L}(\W)}=b^2.
\end{equation}
We will prove well-posedness and exponential stability for system \eqref{modello}, for small initial data, under the assumption

\begin{equation}\label{assumption_delay}
M  b^2e^{\omega\tau} \int_0^t |k(s+\tau)|ds  \leq \gamma+\omega' t, \quad \forall \ t>0,
\end{equation}
for suitable constants $\gamma\geq 0$ and $\omega'\in[0,\omega).$ 

First, we will give an exponential decay result for the abstract model \eqref{abstract_form} under a suitable well-posedness assumption.

\begin{Theorem}\label{generaleCV}
Assume \eqref{assumption_delay}.
Moreover, suppose that 
\begin{itemize}
\item[{(I)}] there exist $\rho>0$, $C_\rho>0$,  with $L(C_\rho)<\frac{\omega-\omega '}{M}$ such that if $U_0 \in \W$ and if $f\in C([-\tau ,0];\W)$ satisfy
\begin{equation}\label{well-posedness}
||U_0||^2_{\W}+ \int_0^{\tau} |k(s)| \cdot ||f(s-\tau)||^2_{\W} ds <\rho ^2,
\end{equation}
then the system \eqref{abstract_form} has a unique solution $U\in C([0,+\infty);\W)$ satisfying $||U(t)||_{\W}\leq C_\rho$ for all $t>0$.
\end{itemize}
Then, for every solution $U$ of \eqref{abstract_form}, with initial data $(U_0, f)$ satisfying \eqref{well-posedness},
\begin{equation}
\label{stimaesponenziale}
||U(t)||_{\W}\leq Me^\gamma \left (||U_0||_{\W}+\int_0^{\tau} e^{\omega s} |k(s)|\cdot ||f(s-\tau)||_{\W} ds\right )e^{-(\omega -\omega'-ML(C_{\rho}))t}, 
\end{equation}
for any $t\geq 0$.
\end{Theorem}
\begin{proof}
From Duhamel's formula we obtain
$$
\begin{array}{l}
\displaystyle{||U(t)||_\W\leq Me^{-\omega t} ||U_0||_\W+Me^{-\omega t} \int_0^t e^{\omega s}|k(s)| \cdot ||\mathcal BU(s-\tau)||_\W ds}\\
\hspace{1.7 cm}
\displaystyle{ +ML(C_\rho)e^{-\omega t} \int_0^t e^{\omega s} ||U(s)||_\W ds,}
\end{array}
$$
where we have used the fact that $||F(U(t))||_\W\leq L(C_\rho) ||U(t)||_\W$ for any $t\geq 0$. Then, we get
$$
\begin{array}{l}
\displaystyle{ ||U(t)||_\W\leq Me^{-\omega t} ||U_0||_\W+Me^{-\omega t} \int_0^{\tau} e^{\omega s} |k(s)| \cdot ||f(s-\tau)||_\W ds}\\
\hspace{1.7 cm}
\displaystyle{+ Me^{-\omega t} \int_{\tau}^t e^{\omega s} b^2 |k(s)|\cdot ||U(s-\tau)||_\W  ds +ML(C_\rho)e^{-\omega t} \int_0^t e^{\omega s} ||U(s)||_\W ds.}
\end{array}
$$
By change of variables $s-\tau=z$ we infer
$$
\begin{array}{l}
\displaystyle{ e^{\omega t}||U(t)||_\W \leq M||U_0||_\W+M\int_0^{\tau} e^{\omega s} |k(s)| \cdot ||f(s-\tau)||_\W ds}\\
\hspace{2 cm}
\displaystyle{+M b^2 e^{\omega \tau} \int_0^t  |k(z+\tau)| e^{\omega z} ||U(z)||_\W dz+ ML(C_\rho) \int_0^t e^{\omega s} ||U(s)||_\W ds.}
\end{array}
$$
Now, let us denote $\tilde u (t):=e^{\omega t} ||U(t)||_\W$. Hence,
$$
\begin{array}{l}
\displaystyle{ \tilde u(t)\leq M||U_0||_\W+M \int_0^{\tau} e^{\omega s} |k(s)| \cdot ||f(s-\tau)||_\W ds}\\
\hspace{2 cm}
\displaystyle{+\int_0^t \left[ M b^2 e^{\omega \tau}  |k(s+\tau)| +ML(C_\rho)\right] \tilde u (s) ds.}
\end{array}
$$
Then, Gronwall estimate implies
$$
\tilde u (t)\leq \left( M||U_0||_\W+M\int_0^{\tau} e^{\omega s} |k(s)| \cdot ||f(s-\tau)||_\W ds\right) e^{ M b^2 e^{\omega \tau} \int_0^t |k(s+\tau)| ds  +ML(C_\rho)t}.
$$
Then, by definition of $\tilde u$ and by assumption \eqref{assumption_delay}, we get \eqref{stimaesponenziale}.
\end{proof}
Therefore, in order to have the exponential stability of solutions to \eqref{modello}, we need to show that the well-posedness assumption $(I)$ holds true for system \eqref{modello}. To do this, we introduce the following energy functional associated to system \eqref{modello}:
$$
E(t):=\frac 12 ||u_t(t)||_H^2+\frac 12 ||A^{\frac 12}u(t)||_H^2-\psi(u(t))+\frac 12  \int_{t-\tau}^t |k(s+\tau)|\cdot ||B^* u_t(s)||_{W_2}^2 ds.
$$
We immediately have the following result.
\begin{Proposition}\label{proposition}
Let $u$ be a solution to \eqref{modello}. If $E(t)\geq \frac 14 ||u_t(t)||_H^2$ for any $t\geq 0$, then
\begin{equation}\label{tesi_prop}
E(t)\leq \bar C(t) E(0),
\end{equation}
for any $t\geq 0$, where 
\begin{equation}\label{Cbar}
\displaystyle{
\bar C (t):= e^{2\int_0^t b^2 (|k(s)|+|k(s+\tau)|) ds}}.
\end{equation}
\end{Proposition}
\begin{proof}
Differentiating $E$ in time yields
$$
\begin{array}{l}
\displaystyle{\frac{d}{dt}E(t)=\langle u_t(t),u_{tt}(t)\rangle_H+\langle A^{\frac 12} u(t),  A^{\frac 12} u_t(t)\rangle_H-\langle \nabla \psi(u(t)), u_t(t)\rangle }\\
\hspace{2 cm}
\displaystyle{+\frac 12 |k(t+\tau)|\cdot ||B^*u_t(t)||^2_{W_2}-\frac 12  |k(t)|\cdot ||B^*u_t(t-\tau)||^2_{W_2}.}
\end{array}
$$
By using \eqref{modello}, we obtain 
$$
\begin{array}{l}
\displaystyle{\frac{d}{dt}E(t) =-||C^* u_t(t)||_{W_1}^2- k(t)\langle B^* u_t(t),B^*u_t(t-\tau)\rangle_{W_2}}\\
\hspace{2 cm}
\displaystyle{+\frac 12  |k(t+\tau)|\cdot ||B^*u_t(t)||_{W_2}-\frac 12  |k(t)|\cdot ||B^*u_t(t-\tau)||_{W_2}.}
\end{array}
$$
We use Young inequality in the second term in the previous equality and we get
$$
\begin{array}{l}
\displaystyle{ \frac{d}{dt}E(t) \leq \frac 12 \left( |k(t)|+|k(t+\tau)|\right) ||B^* u_t(t)||_{W_2}^2}\\
\hspace{2 cm}
\displaystyle{\leq 2 b^2 \left( |k(t)|+|k(t+\tau)|\right) \frac 14 ||u_t(t)||_H^2,}
\end{array}
$$ 
where for the last inequality we used \eqref{normab}. Since we have assumed $E(t)\geq \frac 14 ||u_t(t)||_H^2$ for any $t\geq 0$, we obtain
$$
\frac{d}{dt}E(t) \leq 2 b^2 \left( |k(t)|+|k(t+\tau)|\right)  E(t).
$$ 
By Gronwall inequality we then obtain \eqref{tesi_prop}.
\end{proof}

In order to prove the well-posedness assumption $(I)$ for system \eqref{modello}, we need the following two lemmas. 
\begin{Lemma}
\label{lemma1}
Let us consider the system \eqref{abstract_form} with initial data $U_0\in \W$ and  $f\in C([-\tau,0]; \W).$ Then, there exists a unique continuous local solution $U(\cdot)$ defined on a time  interval $[0,\delta)$, with $\delta \le\tau.$ 
\end{Lemma}
\begin{proof}
Since $t\in [0,\tau]$, we can rewrite the abstract system \eqref{abstract_form} as an undelayed problem:
\begin{eqnarray*}
U'(t)&=& \mathcal AU(t)-k(t) f(t-\tau)+F(U(t)), \quad t\in (0, \tau),\\
U(0)&=&U_0.
\end{eqnarray*}
Then, we can apply the classical theory of nonlinear semigroups (see e.g. \cite{Pazy}) obtaining the existence of a unique solution  on a set $[0,\delta)$, with $\delta \le\tau$.
\end{proof}
\begin{Lemma}
\label{lemma2}
Let $U(\cdot)$ be a non-null solution to \eqref{abstract_form} defined on the interval $[0, \delta),$ with $\delta\le\tau.$ Let $h$ be the strictly increasing function appearing in \eqref{stima_h}. Then,
\begin{enumerate}
\item if $h(||A^\frac{1}{2} u_0(0)||_H)<\frac 12,$ then $E(0)>0$;
\item if $h(||A^\frac{1}{2} u_0(0)||_H)<\frac 12$ and
$
h \left( 2 \bar{C}^\frac{1}{2}(\tau) E^\frac{1}{2}(0) \right) <\frac 12,
$
with $\bar{C}(\cdot)$ defined as in \eqref{Cbar}, then
\begin{equation}
\label{stima E dal basso}
\begin{array}{l}
\displaystyle{ E(t)>\frac{1}{4}||u_t(t)||_H^2+\frac{1}{4}||A^\frac{1}{2}u(t)||_H^2 +\frac{1}{4}\int_{t-\tau}^t |k(s+\tau)| \cdot ||B^*u_t(s)||_{W_2}^2 ds}
\end{array}
\end{equation}
for all $t\in[0, \delta)$. In particular,
\begin{equation}\label{J2}
E(t)>  \frac 14 \Vert U(t)\Vert_{\W}^2, \quad \mbox{for all} \ \ t\in [0, \delta).
\end{equation}
\end{enumerate}
\end{Lemma}
\begin{proof}
First, from assumption (H3) on $\psi$ we can infer that
\begin{equation}
\label{assumptionPsi}
\begin{array}{l}
\displaystyle{|\psi(u)|\leq \int_0^1 |\langle \nabla \psi (su),u\rangle | ds} \\
\hspace{1,15 cm}
\displaystyle{\leq  ||A^\frac{1}{2}u||^2_H \int_0^1 h(s||A^\frac{1}{2}u||_H)s ds\leq \frac{1}{2} h(||A^\frac{1}{2}u||_H)||A^\frac{1}{2}u||^2_H.}
\end{array}
\end{equation}
Hence, under the assumption $h (\Vert A^{\frac 12} u_0(0)\Vert_H) < \frac {1} 2,$ we have that
\begin{equation}\label{27luglio}
\begin{array}{l}
\displaystyle{ E(0)=\frac{1}{2}||u_1||_H^2+\frac{1}{2}||A^\frac{1}{2}u_0(0)||_H^2-\psi(u_0(0))+\frac{1}{2}\int_{-\tau}^0 |k(s+\tau)|\cdot ||B^*u_t(s)||^2_{W_2} ds}\\
\hspace{0,9 cm}
\displaystyle{ >\frac{1}{4}||u_1||^2_H+\frac{1}{4}||A^\frac{1}{2}u_0(0)||^2_H +\frac{1}{4} \int_{-\tau}^0 |k(s+\tau)| \cdot ||B^*u_t(s)||^2_{W_2} ds. }\\
\end{array}
\end{equation}
Note that, if the right-hand side of \eqref{27luglio} is zero, then $U(\cdot)$ is the null solution. Therefore, we have proven claim 1.

In order to prove the second statement, we argue by contradiction. Let us denote
$$
r:=\sup \{ s\in [0,\delta) : \eqref{stima E dal basso} \quad \text{holds} \quad \forall t\in [0,s)\}.
$$
We suppose by contradiction that $r<\delta$. Then, by continuity, we have
\begin{equation}
\label{continuita}
\begin{array}{l}
\displaystyle{E(r)=\frac{1}{4}||u_t(r)||^2_H+\frac{1}{4}||A^\frac{1}{2}u(r)||_H^2+\frac{1}{4}\int_{r-\tau}^r |k(s+\tau)| \cdot ||B^*u_t(s)||_{W_2}^2 ds.}
\end{array}
\end{equation}
Now, since from \eqref{continuita}
$$
\frac{1}{4} \Vert A^{\frac 1 2} u(r)\Vert^2_H\leq E(r),
$$
we can infer, by using Proposition \ref{proposition}, that
\begin{equation}\label{risultato}
\begin{array}{l}
\displaystyle{ h(||A^\frac{1}{2}u(r)||_H)\leq h\left( 2 E^\frac{1}{2}(r)\right)  <h\left( 2\bar{C}^\frac{1}{2}(\tau)E^\frac{1}{2}(0)\right) <\frac{1}{2}.}
\end{array}
\end{equation}
Hence, we have that
$$
\begin{array}{l}
\displaystyle{ E(r)=
\frac{1}{2}||u_t(r)||_H^2+\frac{1}{2}||A^\frac{1}{2}u(r)||_H^2-\psi(u(r))+\frac{1}{2}\int_{r-\tau}^r|k(s+\tau)|\cdot ||B^*u_t(s)||^2_{W_2} ds}\\
\hspace{0.9 cm}
 \displaystyle{
>\frac{1}{4}||u_t(r)||_H^2+\frac{1}{4}||A^\frac{1}{2}u(r)||_H^2+\frac{1}{4}\int_{r-\tau}^r|k(s+\tau)| \cdot ||B^*u_t(s)||_{W_2}^2 ds,}
\end{array}
$$
where in the last estimate we used \eqref{assumptionPsi} and \eqref{risultato}. This contradicts the maximality of $r$. Hence, $r=\delta$ and this concludes the proof of the lemma.
\end{proof}
We are now ready to prove the well-posedness assumption for system \eqref{modello}.
\begin{Theorem}\label{wellpos_thm}
If hypothesis \eqref{assumption_delay} is satisfied, then problem \eqref{abstract_form}, with initial data $U_0\in \W$ and $f\in C([-\tau,0];\W),$  satisfies the well-posedness assumption (I). Hence, for solutions of \eqref{abstract_form} corresponding to sufficiently small initial data the exponential decay estimate \eqref{stimaesponenziale} holds true.
\end{Theorem}
\begin{proof}
Let us fix $N\in\nat$
such that
\begin{equation}\label{stimaN}
\begin{array}{l}
\displaystyle{C_N:=2M^2e^{2\gamma} \left( 1+K   e^{\omega \tau} b^2 \right) \left(1+e^{2\omega \tau}K\right) e^{-(\omega-\omega')N\tau}\leq 1.}
\end{array}
\end{equation}
Moreover let $\rho>0$ be such that
$$
\rho\leq \frac{1}{2\bar C (N\tau)}h^{-1}\left( \frac 12\right),
$$
and assume
\begin{equation}\label{dis1}
||u_1||_H^2+||A^{\frac 12}u_0(0)||_H^2 +\int_{-\tau}^0 |k(s+\tau)| \cdot ||g(s)||_{W_2}^2 ds \leq \rho^2.
\end{equation}
We observe that this is equivalent to (by considering the abstract formulation \eqref{abstract_form})
$$
||U_0||_{\W}^2 +\int_0^{\tau} |k(s)| \cdot ||f(s-\tau)||_{\W}^2 ds \leq \rho^2.
$$
First of all, from Lemma \ref{lemma1} we know that there exists a solution $u$ to \eqref{modello} on the time interval $[0,\delta)$ with $\delta\leq \tau$. Now, we have that
$$
h(||A^{\frac 12} u_0(0)||_H)<h(\rho) \leq h\left( \frac{1}{2\bar C ^{\frac 12} (N\tau)} h^{-1}\left( \frac 12\right)\right) <\frac 12,
$$
where we have used the fact that $\bar C (N\tau) >1$. 
 Hence, by Lemma \ref{lemma2}, $E(0)>0$. Moreover, from \eqref{assumptionPsi} we get
$$
\begin{array}{l}
\displaystyle{E(0)\leq \frac 12 ||u_1||_H^2+\frac 34 ||A^{\frac 12} u_0(0)||_H^2+\frac 12 \int_{-\tau}^0 |k(s+\tau)| \cdot ||g(s)||_{W_2}^2 ds\leq \rho^2,}
\end{array}
$$
which gives us
$$
h\left(2\bar C^{\frac 12} (N\tau)E^{\frac 12} (0)\right)< h\left(2\bar C^{\frac 12} (N\tau) \rho\right) <h\left(h^{-1}\left(\frac 12\right)\right) =\frac 12.
$$
Since $\bar C (N\tau)\ge \bar C (\tau),$ then 
\begin{equation}\label{29luglio2}
h\left(2\bar C^{\frac 12} (\tau)E^{\frac 12} (0)\right)\le h\left(2\bar C^{\frac 12} (N\tau)E^{\frac 12} (0)\right)<\frac 12.
\end{equation}
Hence, we can apply again Lemma \ref{lemma2} and we can infer that \eqref{J2} is satisfied for any $t\in[0,\delta)$. Then, we can use Proposition \ref{proposition} getting
\begin{equation}\label{29luglio}
\begin{array}{l}
\displaystyle{0<\frac 14 ||u_t(t)||_H^2+\frac 14 ||A^{\frac 12} u(t)||_H^2 }\\
\hspace{0.4 cm}
\displaystyle{+\frac 14  \int_{t-\tau}^t |k(s+\tau)| \cdot ||B^*u_t(s)||_{W_2}^2 ds< E(t)\leq  \bar C(\tau)E(0),}
\end{array}
\end{equation}
for any $t\in [0,\delta)$. Then, we can extend the solution in $t=\delta$ and on the entire interval $[0,\tau]$. Now, from \eqref{29luglio} and \eqref{29luglio2}, we obtain
$$
h(||A^{\frac 12} u(\tau)||_H)\le h(2E(\tau))\le h \left (2\bar C^{\frac 12}(\tau)E^{\frac 12}(0)\right ) <\frac 12.
$$
Therefore, there exists $\delta'\in(0,\tau]$ such that
$$
h(||A^{\frac 12} u(t)||_H)<\frac 12, \quad \forall \ t\in [\tau,\tau+\delta').
$$
So, arguing analogously to the proof of  Lemma \ref{lemma2}, one can obtain
$$E(t)>\frac 14 \Vert u_t(t)\Vert_H^2 + \frac 14 \Vert A^{\frac 12}u(t)\Vert_H^2+\frac 14\int_{t-\tau}^t \vert k(s+\tau)\vert \cdot \Vert B^*u_t(s)\Vert_{W_2}^2 ds,$$
for all $t\in [\tau,\tau+\delta').$ Then, 
$$
E(t)>\frac 14 ||u_t(t)||_H^2, \quad \forall \ t \in [\tau,\tau+\delta').
$$
Hence, we can apply once again Proposition \ref{proposition} which yields
$$
\begin{array}{l}
\displaystyle{0<\frac 14 ||u_t(t)||_H^2 +\frac 14 ||A^{\frac 12} u(t)||_H^2}\\
\hspace{0.4 cm}
\displaystyle{+\frac 14 \int_{t-\tau}^t |k(s+\tau)|\cdot ||B^*u_t(s)||_{W_2}^2 ds<E(t)\leq \bar C (2\tau) E(0)},
\end{array}
$$
where we used the fact that $\delta'\leq \tau$ and the monotonicity of $\bar C(\cdot)$. Then, as before, we can extend  the solution to the interval $[0,2\tau].$ Iterating this procedure, one can extend the solution  to the whole time interval $[0,N\tau]$ with $N$ satisfying \eqref{stimaN}. Now for $t=N\tau$ we have that
$$
\begin{array}{l}
\displaystyle{h(||A^{\frac 12} u(N\tau)||_H) \leq h(2E^{\frac 12} (N\tau))}\\
\hspace{3 cm}
\displaystyle{\leq h(2\bar C ^{\frac 12} (N\tau)E^{\frac 12} (0))\leq h(2\bar C ^{\frac 12} (N\tau)\rho)<\frac 12.}
\end{array}
$$
Moreover,
$$
\frac 14 ||U(t)||_{\W}^2\leq E(t)\leq \bar C (N\tau_{min})E(0)<\bar C(N\tau)\rho^2, \quad \forall \ t \in [0,N\tau],
$$
and so 
$$
||U(t)||_{\W}\leq C_\rho:= 2\bar C^{\frac 12} (N\tau)\rho,
$$
for any $t\in[0,N\tau]$. Now, eventually choosing smaller values of $\rho$, we suppose that $\rho$ is such that $L(C_\rho)<\frac{\omega-\omega'}{2M}$. Therefore, assumption $(I)$ is satisfied along the interval $[0,N\tau]$. Hence, Theorem \ref{generaleCV} gives us the following estimate:
\begin{equation}\label{1}
||U(t)||_{\W}\leq Me^{\gamma}  \left (||U_0||_{\W}+ \int_0^{\tau} e^{\omega s} |k(s)|\cdot ||f(s-\tau)||_{\W} ds\right )e^{-\frac{\omega -\omega'}{2}t},
\end{equation}
for any $t\in [0,N\tau]$.
By using assumption \eqref{damp_coeff} and H\"older inequality we get
$$
\begin{array}{l}
\displaystyle{ \int_0^{\tau} |k(s)|e^{\omega s} ||f(s-\tau)||_\W ds \leq e^{\omega \tau} \left( \int_0^{\tau} |k(s)| ds \right)^{\frac 12} \left( \int_0^{\tau} |k(s)|\cdot ||f(s-\tau)||_\W ^2 ds\right)^{\frac 12}}\\ \medskip
\hspace{4.65 cm}
\displaystyle{ \leq e^{\omega \tau} {K}^{\frac 12} \rho.}
\end{array}
$$
Therefore, from \eqref{1} we get
$$
||U(N\tau)||_\W^2\leq 2M^2\rho^2e^{2\gamma}\left( 1+e^{2\omega \tau}  K\right) e^{-(\omega-\omega')N\tau}.
$$
Moreover, if $s\in [N\tau,N\tau+\tau]$, then $ s-\tau\in[N\tau-\tau,N\tau]$. Then,
$$
\begin{array}{l}
\displaystyle{||U(N\tau)||_\W^2 + \int_{N\tau}^{N\tau+\tau} e^{\omega(s-N\tau)} |k(s)|\cdot ||B^* u_t(s-\tau)||_{W_2}^2 ds\leq  C_N\rho^2\leq \rho^2,}
\end{array}
$$
where we have used \eqref{stimaN}. Hence, we can infer that
$$
||U(N\tau)||_\W^2+\int_{N\tau}^{N\tau+\tau} |k(s)|\cdot ||B^* u_t(s-\tau)||_{W_2}^2 ds \leq \rho^2.
$$
We can proceed by applying a similar argument shown before on the interval $[N\tau, 2N\tau]$, obtaining a solution on the interval $[0,2N\tau]$. Iterating the process, we find a unique global solution to \eqref{abstract_form} satisfying the well-posedness assumption $(I)$. Hence, the theorem is proved. 
\end{proof}

\section{Examples}
\label{Examples}\hspace{5mm}

\setcounter{equation}{0}
In this section we give some applications of previous abstract well-posedness and stability results. We will show that the following systems can be rewritten in the abstract form \eqref{abstract_form} and so, under suitable assumptions, global existence and stability decay estimates hold for small initial data. 
\subsection{Wave equation with damping and source term}
Let $\Omega$ be an open subset of $\RR^d$, with boundary $\Gamma$ of class $C^2$ and let $\mathcal O \subset \Omega$ be an open subset which satisfies the geometrical control property in \cite{Bardos}. For instance, denoting by $m$ the standard multiplier $m(x)=x-x_0,$ $x_0\in \RR^d,$ as in \cite{Lions}, $\mathcal O$ can be the intersection of $\Omega$ with an open neighborhood of the set 
$$\Gamma_0=\left\{ \,x\in\Gamma\ :\ m(x)\cdot \nu(x)>0\,\right\}.$$
 Moreover, let $\tilde{\mathcal{O}}\subset\Omega$ be another open subset. Denoting by $\chi_D$ the characteristic function of a set $D,$ we consider the following wave equation
\begin{equation}\label{wave}
\begin{array}{l}
\displaystyle{u_{tt}(x,t)-\Delta u(x,t)+a\chi_{\mathcal{O}}(x) u_t(x,t)+k(t)\chi_{\tilde{\mathcal O}}(x)u_t(x,t-\tau)}\\
\hspace{7 cm}
\displaystyle{=u(x,t)|u(x,t)|^\beta, \quad (x,t)\in\Omega\times (0,+\infty),}\\
\displaystyle{u(x,t)=0, \quad (x,t)\in\Gamma\times (0,+\infty),}\\
\displaystyle{u(x,0)=u_0(x), \quad x\in\Omega,}\\
\displaystyle{u_t(x,s)=u_1(x,s) \quad (x,s)\in\Omega\times [-\tau,0],}
\end{array}
\end{equation}
where $a$ is a positive constant, $\tau>0$ is the time delay, $\beta >0$ and the delayed damping coefficient $k(\cdot):[0,+\infty)\to (0,+\infty)$ is a $L^1_{loc}([0,+\infty))$ function satisfying \eqref{damp_coeff}. System \eqref{wave} falls in the form \eqref{modello} with $A=-\Delta$ with dense domain $D(A)=H^2(\Omega)\cap H_0^1(\Omega)$. Setting $v(t)=u_t(t)$ and $U(t)=(u(t),v(t))^T$ for any $t\geq 0$, we can rewrite system \eqref{wave} in the abstract form \eqref{abstract_form}, with $\W=H_0^1(\Omega)\times L^2(\Omega)$,
$$
\mathcal A=\begin{pmatrix}
0 & Id \\
\Delta & -a\chi_{\mathcal O}
\end{pmatrix}
$$
and $\mathcal B$ and $F$ defined as
$$
\mathcal B \begin{pmatrix} u \\ v \end{pmatrix} = \begin{pmatrix} 0 \\ -\chi_{\tilde{\mathcal{O}}} v \end{pmatrix}, \qquad F(U(t))= \begin{pmatrix} 0 \\  u(t)|u(t)|^\beta \end{pmatrix}, \quad \forall\ t \geq 0.
$$
We know that $\mathcal{A}$ generates an exponentially stable $C_0$-semigroup $\{S(t)\}_{t\geq 0}$ (see e.g. \cite{K}), namely there exist $\omega,M>0$ such that
$$
||S(t)||_{\mathcal{L}(\mathcal H)} \leq Me^{-\omega t}, \quad \forall \ t\geq 0.
$$
Moreover, we consider the following functional:
$$\psi(u)=\frac{1}{\beta+2}\int_{\Omega} |u(x)|^{\beta+2} dx,$$
for any $u\in H_0^1(\Omega)$. For any $\beta\in \left( 0, \frac{4}{d-2}\right)$, $\psi$ is well-defined by Sobolev's embedding theorem. Furthermore, $\psi$ is G\^ateaux differentiable for any $u\in H_0^1(\Omega)$ with G\^ateaux derivative given by
$$
D\psi(u)(v)=\int_{\Omega} |u(x)|^\beta u(x)v(x)dx, \quad \forall \ v\in H_0^1(\Omega).
$$
As in \cite{ACS}, it is possible to show that $\psi$ satisfies hypotheses (H1), (H2) and (H3), provided that $\beta \in \left( 0, \frac{2}{d-2}\right]$. We define the following energy functional for any $t\geq 0$:

$$E(t):= \frac 12 \int_\Omega |u_t(x,t)|^2 dx+\frac 12 \int_{\Omega} |\nabla u(x,t)|^2 dx-\psi(u(x,t))+\frac 12 \int_{t-\tau}^t\int_{\tilde{\mathcal O}} |k(s+\tau)|\cdot |u_t(x,s)|^2 dx ds.
$$
Then, Theorem \ref{generaleCV} can be applied to system \eqref{wave}, under the assumption \eqref{assumption_delay}, obtaining well-posedness and stability results for small initial data.

\begin{Remark}{\rm
We want to emphasize that our result holds true for every pair of subsets $({\mathcal{O}}, \tilde{\mathcal{O}})$ of $\Omega.$ The only condition that we require is that 
${\mathcal{O}}$ satisfies a geometric control property. On the contrary, in \cite{JEE15, NicaisePignotti18, KP}, in the nonlinear setting, it is required $\tilde{\mathcal{O}}\subseteq {\mathcal{O}}.$ Same remark applies to the following example.
}
\end{Remark}
\subsection{Plate equation with damping and source term} 
Let $\Omega\subset\RR^d$ be an open subset with boundary $\Gamma$ of class $C^2$ and let $\mathcal O \subset \Omega$ be an open subset which satisfies the geometrical control property in \cite{Bardos}. Let $\tilde{\mathcal O}$ be another subset of $\Omega$. We consider the following plate equation:
\begin{equation}\label{plate}
\begin{array}{l}
\displaystyle{u_{tt}(x,t)+\Delta^2 u(x,t)+a\chi_{\mathcal{O}}(x) u_t(x,t)+k(t)\chi_{\tilde{\mathcal O}}(x)u_t(x,t-\tau)}\\
\hspace{7 cm}
\displaystyle{=u(x,t)|u(x,t)|^\beta, \quad (x,t)\in\Omega\times (0,+\infty),}\\
\displaystyle{u(x,t)=\frac{\partial u}{\partial \nu}(x,t)=0, \quad (x,t)\in\Gamma\times (0,+\infty),}\\
\displaystyle{u(x,0)=u_0(x), \quad x\in\Omega,}\\
\displaystyle{u_t(x,s)=u_1(x,s) \quad (x,s)\in\Omega\times [-\tau,0],}
\end{array}
\end{equation}
with $a, k(\cdot), \tau$ as before and $\beta>0.$ 
As for the previous case, one can rewrite system \eqref{plate} in the abstract form \eqref{abstract_form}. Now, since the nonlinear term satisfies hypotheses (H1)-(H2)-(H3) for $(d-4)\beta\le 4$ (cf. e.g. \cite{Mustafa}), then we can apply Theorem \ref{generaleCV} for $\beta$ in this range. Therefore, under assumption \eqref{assumption_delay}, for the model \eqref{plate} we have well-posedness and stability results for small initial data.

\end{document}